\documentclass[12pt,oneside,reqno]{amsart}

\usepackage[T1]{fontenc}
\usepackage{lmodern}

\usepackage{amssymb,amsfonts,amsmath,amsthm,amscd,mathtools,cancel,latexsym}

\usepackage{graphicx,float,wrapfig}

\usepackage{caption,subcaption}

\usepackage{geometry}

\usepackage{xcolor}

\usepackage{tikz}
\usetikzlibrary{cd,arrows,arrows.meta,graphs}
\usepackage[all]{xy}

\usepackage{enumitem}

\usepackage{chngcntr}

\usepackage{url}

\def\qed{\ifhmode\textqed\fi
	\ifmmode\ifinner\hfill\quad\qedsymbol\else\dispqed\fi\fi}
\def\textqed{\unskip\nobreak\penalty50
	\hskip2em\hbox{}\nobreak\hfill\qedsymbol
	\parfillskip=0pt \finalhyphendemerits=0}
\def\dispqed{\rlap{\qquad\qedsymbol}}

\setlist[enumerate]{%
	label={\normalfont\arabic*.},
	left={\parindent},
	itemsep={5pt}%
}%

\newcommand{\MBB}{\mathbb}
	\def\ZZ{\MBB Z} 

\newcommand{\MC}{\mathcal}
	\def\cC{\MC C}
	\def\dD{\MC D}

\newcommand{\MRM}{\mathrm}
	\def\m{\MRM m}

\let\union=\cup

\let\sect=\cap
\let\dirsum=\oplus

\let\isom=\cong

\let\Sect=\bigcap

\let\epsilon=\varepsilon
\let\phi=\varphi

\let\to=\rightarrow

\def\Implies{\ifmmode\Longrightarrow \else
	\unskip${}\Longrightarrow{}$\ignorespaces\fi}
\def\implies{\ifmmode\Rightarrow \else
	\unskip${}\Rightarrow{}$\ignorespaces\fi}
\def\iff{\ifmmode\Longleftrightarrow \else
	\unskip${}\Longleftrightarrow{}$\ignorespaces\fi}

\let\:=\colon

\DeclareMathOperator{\depth}{depth}

\DeclareMathOperator{\Hom}{Hom}

\DeclareMathOperator{\Ext}{Ext}

\DeclareMathOperator{\Cone}{Cone}

\theoremstyle{plain}{%
	\newtheorem{Theorem}{Theorem}[section]
	\newtheorem{Lemma}[Theorem]{Lemma}
	\newtheorem{Proposition}[Theorem]{Proposition}
	\newtheorem{Corollary}[Theorem]{Corollary}
	\newtheorem*{Question}{Question}	

}%

\theoremstyle{definition}{%
	\newtheorem{Definition}[Theorem]{Definition}
	\newtheorem{Example}[Theorem]{Example}
	
	\newtheorem{Remark}[Theorem]{Remark}
}%

\def\CM{{Cohen-Macaulay}}
\def\SCM{{sequentially Cohen-Macaulay}}

\begin{document}
\title{Sequentially Cohen-Macaulay binomial edge ideals}
\author{Ernesto Lax, Giancarlo Rinaldo, Francesco Romeo}

\address{Ernesto Lax, Department of mathematics and computer sciences, physics and earth sciences, University of Messina, Viale Ferdinando Stagno d'Alcontres 31, 98166 Messina, Italy}
\email{erlax@unime.it}

\address{Giancarlo Rinaldo, Department of mathematics and computer sciences, physics and earth sciences, University of Messina, Viale Ferdinando Stagno d'Alcontres 31, 98166 Messina, Italy}
\email{giarinaldo@unime.it}

\address{Francesco Romeo, Department of Electrical and Information Engineering ``Maurizio Scaranò'', University of Cassino and Southern Lazio, 03043 Cassino, Italy}
\email{francesco.romeo@unicas.it}

\subjclass{Primary 13C14, 13F70, 05C25, 05C40}
\keywords{Binomial edge ideals, cycles, wheels, block graphs, cones, modules of deficiency, sequentially Cohen-Macaulay}

\dedicatory{To J{\"u}rgen Herzog (1941-2024)}

\begin{abstract}
	We prove that wheels and block graphs have sequentially Cohen-Macaulay binomial edge ideals. Moreover, we provide a construction of new families of sequentially Cohen-Macaulay graphs by cones.
\end{abstract}

\maketitle

\section*{Introduction}
Let $G$ be a simple graph with the vertex set $[n]$ and the edge set $E(G)$. The notion of binomial edge ideals was introduced by Herzog et al. in \cite{HHHKR}, and independently by Ohtani in \cite{O}. Many algebraic properties and invariants of such ideals were described \textit{via} the combinatorics of the underlying graph, see \cite{HHO_book} for a nice survey on this topic. The most important one is the so-called cutset property. Let $G\setminus T$ denote the graph obtained by removing the vertices of $T\subset [n]$. $T$ is called a {\em cutset} if the number of connected components of $G\setminus T$, denoted by $c(T)$, satisfies $c(T)>c(T\setminus\{v\})$ for any $v\in T$. We denote by $\cC(G)$ the set of all cutsets for $G$. For example, the Krull dimension of the ring induced by the binomial edge ideal is 
\begin{equation}\label{Eq:dimension}
	d=\max \{n +c(T) - |T|\ :\ T\in \cC(G)\}.
\end{equation}

One of the main topics of the study of binomial edge ideals is to classify the {\CM} ones. As an example, after their introduction, the first paper is \cite{EHH} where the main result is to classify the block graphs that are {\CM}. After that, many authors focused on the problem (see \cite{BNB},\cite{BMS2018},\cite{EHH},\cite{RR}, \cite{RS}). Recently, Gorenstein binomial edge ideals (\cite{G-M}) and licci binomial edge ideals (\cite{ERT2020}) have been characterized. 
A full classification of graphs whose binomial edge ideals are {\CM} is not yet known. However, significant progress in this direction has been recently made in \cite{BMRS2024},\cite{BMRS2025},\cite{BMS2022},\cite{KSM2012},\cite{KSM2015},\cite{LMRR2023},\cite{SS2023}. A nice and deeply studied generalization of the {\CM} ring is the {\SCM} one (see \cite{Stan_book}).

As in the case of the {\CM} property, for studying {\SCM} binomial edge ideals, one can reduce to connected indecomposable graphs as proved in \cite{ERT2022b}. In particular, in \cite{ERT2022b}, the authors classify the binomial edge ideals with quadratic Gröbner bases that are {\SCM}, that is those ideals which are associated with closed graphs (see \cite{EHH}). Two are the main ingredients to obtain such result: 1) a characterization of Goodarzi \cite{G} on {\SCM} homogeneous ideals; 2) the initial ideal of closed graphs. In \cite{SZ}, the authors classified the complete bipartite graphs that are {\SCM} studying the deficiency modules (see \cite{S_book}).

In Section 1, we focus on the graphs with one cutset. In particular we use the Goodarzi's filtration and the deficiency modules. An interesting observation of this section is obtained by the following invariant
\[
\m(G)=\min\{n+c(T)-|T| \ :\ T\in \cC(G)\}.
\]
If $\m(G) = n+1$, namely $\m(G)=\dim S/P_{\emptyset}$, and $J_G$ is {\SCM}, then $G$ has at least a cutpoint, namely a cutset of cardinality 1 (see Proposition \ref{prop:SCM-mindim_onecutvertex}).

In Section \ref{sec:2}, we focus on the family of graphs that are blocks, namely graphs without cutpoints. Hence, in this case $\m(G)\leq n$ and we prove that cycles ($\m(G)=n$) and wheels ($\m(G)=n-1$) are {\SCM} by using the Goodarzi's filtration (see also \cite{Z}).

In Section \ref{sec:3}, we prove that block graphs are {\SCM}. We underline that in general the filtration of Goodarzi gives us new ideals, induced by the primary decomposition of $J_G$, that are not anymore binomial edge ideals. Nevertheless, using elementary containment problem (see Lemma 3.1) by good properties of the block graphs, it is possible to control the filtration using the Mayer--Vietoris sequence (see Theorem 3.2).

In \cite{RR}, the authors provide a useful tool to generate new Cohen-Macaulay graphs using cones. As an application, this has been used to completely classify all bipartite {\CM} graphs \cite{BMS2018}. In Section \ref{sec:4}, we extend this result to {\SCM} property, obtaining that a cone is {\SCM} if and only if its components are {\SCM} (see Theorem \ref{thm:cones}).

\section{Graphs with only one non--empty cutset}\label{sec:1}
Let $G$ be a graph with the vertex set $V(G)=[n]=\{1,\ldots,n\}$ and the edge set $E(G)$. A subset $C\subset V(G)$ is called a {\em clique} of $G$ if for all $i,j \in C$ with $i\neq j$ one has $\{i,j\} \in E(G)$. A {\em free vertex} of $G$ is a vertex that belongs to exactly one maximal clique of $G$, with respect to inclusion. A {\em block} of $G$ is a connected subgraph of $G$ that has no cutpoints, which is maximal with respect to this property. A {\em block graph} is a graph in which every block is a complete graph. A connected graph $G$ is {\em decomposable} if there exist two subgraphs $G_1, G_2$ of $G$ such that $G=G_1 \union G_2$ and $V(G_1)\sect V(G_2)=\{v\}$, where $v$ is a free vertex of $G_1$ and $G_2$.

Let $K$ be a field and $S=K[x_1,\ldots,x_n,y_1,\ldots,y_n]$ be the standard graded polynomial $K$-algebra in $2n$ indeterminates. The {\em binomial edge ideal} of $G$ is the ideal $J_G$ of $S$ generated by all binomials $f_{ij}=x_iy_j-x_jy_i$, such that $i<j$ and $\{i,j\}\in E(G)$.

\begin{Definition}\label{def:SCM_definition}
    A $S$-module $M$ is called {\em\SCM} if there exist a finite filtration of submodules of $M$
    \begin{equation*}
        0 = M_0 \subset M_1 \subset \cdots \subset M_{d-1} \subset M_r = M   
    \end{equation*}
    such that each quotient $M_i/M_{i-1}$ is {\CM} and $\dim M_i/M_{i-1} < \dim M_{i+1}/M_i$ for all $i$.
\end{Definition}

A filtration satisfying the above conditions is called a {\em{\SCM} filtration} of $M$. It is easy to see that if the {\SCM} filtration exists, then it is unique (see \cite{S}, \cite{Stan_book}). One can easily see that any {\CM} module $M$ is {\SCM}, with {\SCM} filtration $0=M_0 \subset M_1 = M$. Hence, one has\par\medskip

\noindent
\begin{center}
	{\CM}\quad {$\Implies$}\quad {\SCM},
\end{center}\medskip
but the converse is not true in general.\par\medskip

Constructing the {\SCM} filtration for a module $M$ may not be easy. Nevertheless many authors in the years described different ways to characterize the {\SCM} property for homogeneous ideals/modules.

We first focus on a homological characterization of the {\SCM} property due to Schenzel (see \cite{S_book}), involving the notion of modules of deficiency.

\begin{Definition}
	Let $M$ be a finitely generated graded $S$-module and put $d=\dim{M}$. For an integer $i\in\ZZ$ the {\em $i$th module of deficiency} of $M$ is defined as it follows
	\[
	\omega^{i}(M)=\Ext^{2n-i}_S(M,S(-2n)).
	\]
	If $i=d$, the module $\omega^{d}(M)=\omega(M)$ is the {\em canonical module} of $M$.
\end{Definition}

The modules of deficiency were introduced and studied in \cite{S_book}. Note that, by the graded version of Local Duality, the $i$th module of deficiency is the Matlis dual of the $i$th local cohomology module, hence the following graded isomorphism holds
\[
\omega^{i} (M) \isom \Hom_K(H^i(M),K),\quad\text{for all $i\in\ZZ$},
\]
where $H^i(M)$ is the $i$th local cohomology module of $M$. From the given relation, since $\depth M = \min\{i : H^i(M)\neq 0\}$ and $\dim M = \max\{i : H^i(M)=0\}$,
it follows that
\[
\omega^i(M) = 0,\quad \text{for all $i<\depth M$ or $i>\dim M$},
\]
which means that, as the local cohomology modules, the modules of deficiency vanish outside the range between $\depth M$ and $\dim M$.

The following results state important properties of the modules of the deficiency (see \cite{SZ}) that will be useful throughout the section.

\begin{Proposition}\label{1:DeficiencyModules-Prop}
	Let $M$ be a finitely generated graded $S$-module and let $d=\dim M$. Then the following holds:
	\begin{enumerate}
		\item $\dim \omega^i(M) \leq i$ and $\dim \omega(M)=d$;
		\item there is a natural homomorphism $M\to \omega(\omega(M))$, which becomes a module isomorphism if and only if $M$ satisfies the Serre condition $S_2$;
	\end{enumerate}
\end{Proposition}

\begin{Theorem}\label{1:SCM_DefModChar}
	Let $M$ be a finitely generated graded $S$-module, with $d=\dim M$. Then the followings are equivalent:
	\begin{enumerate}
		\item $M$ is a {\SCM} module;
		\item for all $0\leq i \leq d$ the $i$th module of deficiency $\omega^i(M)$ is either zero or an $i$-dimensional {\CM} module.
	\end{enumerate}
\end{Theorem}

From now on, throughout the paper we will use the following Mayer--Vietoris sequence
\begin{equation}\label{eq:GeneralMayerVietoris}
0 \to S/J_{G} \to S/J_{G_v} \oplus S/(J_{G\setminus \{v\}} + (x_v, y_v)) \to S/(J_{G_{v}\setminus \{v\}} + (x_v, y_v)) \to 0,
\end{equation}
where $G\setminus \{v\}$ denotes the induced subgraph of $G$ on the vertex set $V(G)\setminus\{v\}$, $G_{v}$ is the graph on the vertex set $V(G)$ and edge set $E(G)\union \{\{u,v\} : u,v\in N_G (v)\}$, where $N_G (v)$ is the neighborhood of $v$ in $G$.

We begin our investigation by focusing on a {\em block star graphs} $G$, namely a block graph with only one cutpoint (see Figure \ref{fig:blockstar}). Assume $G$ has $t$ cliques attached to the cutpoint $v$. Such a graph has $\cC(G)=\{\emptyset,\{v\}\}$, thus
\[
J_{G_v} = P_{\emptyset} = J_{K_n},\qquad J_{G\setminus \{v\}} + (x_v,y_v)=P_{\{v\}}.
\]

\noindent
Moreover, $J_{G_v \setminus \{v\}} + (x_v,y_v) = J_{G_v} + (J_{G\setminus \{v\}} + (x_v,y_v)) = J_{K_{n-1}} + (x_v, y_v)$.

In the following we  also assume that $t\geq 3$. In fact, if $t=2$, then $J_G$ is {\CM} since $G$ is a decomposable graph (see \cite{ERT2022b}). With this assumption, by \eqref{Eq:dimension} we have
\[
\dim S/J_G = n+t-1.
\]

\begin{figure}[H]
	\centering
	\begin{tikzpicture}[scale=.9,vertices/.style={draw, fill=black, circle, inner sep=1pt}]
		\node[vertices] (v) at (0,0) {}; 
		\node[vertices] (1) at (-1,1) {}; 
		\node[vertices] (2) at (0,1) {}; 
		\node[vertices] (3) at (1,1) {};		
		\node[vertices] (4) at (1,0) {}; 
		\node[vertices] (5) at (1,-1) {};
		\node[vertices] (6) at (0,-1) {};		
		\node[vertices] (7) at (-1,-1) {}; 
		\node[vertices] (8) at (-1,0) {};
		
		\foreach \k in {1,...,8}{
		    \draw (v)--(\k);
		};
		
		\draw (2)--(3);
		
		\draw (4)--(5);
		\draw (4)--(6);
		\draw (5)--(6);
		
		\draw (7)--(8);
		
		\node at (.5,0.2) {$v$};
	\end{tikzpicture}
	\hspace*{2cm}
	\begin{tikzpicture}[scale=.9,vertices/.style={draw, fill=black, circle, inner sep=1pt}]
			\node[vertices] (v) at (0,1.54) {};
			
			\node[vertices] (1) at (-.5,0) {};
			\node[vertices] (2) at (.5,0) {};
			\node[vertices] (3) at (.81,.95) {};
			\node[vertices] (4) at (-.81,.95) {};
			
			\node[vertices] (5) at (0,2.54) {};
			\node[vertices] (6) at (-1,2.54) {};
			\node[vertices] (7) at (-1,1.54) {};
			
			\node[vertices] (8) at (1.31,1.19) {};
			\node[vertices] (9) at (.96,2.49) {};
		
			\node at (0.17,1.9) {$v$};
			\foreach \k in {1,...,9}{
			    \draw (v)--(\k);
			};
			\draw (1)--(2) (1)--(3) (1)--(4) (2)--(3) (2)--(4) (3)--(4);
			\draw (5)--(6) (5)--(7) (6)--(7);
			\draw (8)--(9); 
	\end{tikzpicture}
	\caption{Examples of block star graphs}
	\label{fig:blockstar}
\end{figure}
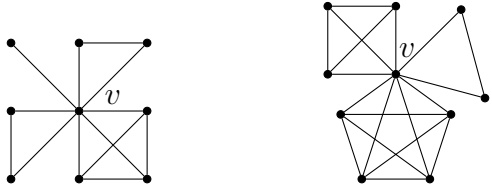

\begin{Theorem} \label{1:BlockStar_DefModDescr+SCM}
	Let $G$ be block star graph on $n$ vertices with $t\geq 3$ cliques attached to a single vertex $v$ and let $J_G$ be the binomial edge ideal of $G$. Then
	\begin{enumerate}
		\item $\omega^{n+1}(S/J_G)$ is a {\CM} module of dimension $n+1$, moreover, we have $\omega^{n+1}(\omega^{n+1}(S/J_G))\isom (J_{K_{n-1}} + (x_v,y_v))/J_{K_n}$;
		\item $\omega(S/J_G)\isom\omega(S/(J_{G\setminus\{v\}}+(x_v,y_v)))$;
		\item $\omega^{i}(S/J_G)=0$ for all $i\neq n+1,n+t-1$.
	\end{enumerate}
	In particular, $J_G$ is {\SCM}.
\end{Theorem}
\begin{proof}
	We will show the statement by applying Theorem \ref{1:SCM_DefModChar}. In the following we will proceed similarly to the proof of \cite[Theorem 4.1]{SZ}.

	Consider the Mayer--Vietoris sequence
	\begin{equation}\label{1:BEI-BlockStar-Seq}
		0 \to S/J_G \to S/J_{K_n} \dirsum S/(J_{G\setminus\{v\}}+(x_v,y_v)) \to S/(J_{K_{n-1}}+(x_v,y_v)) \to 0.
	\end{equation}	
	
	We now apply the local cohomology functors to the sequence \eqref{1:BEI-BlockStar-Seq}. Since $S/J_{K_n}$ is a {\CM} module of dimension $n+1$, we have that $H^{i}(S/J_{K_n})=0$ for all $i\neq {n+1}$ and, since $S/(J_{G\setminus\{v\}}+(x_v,y_v))$, $S/(J_{K_{n-1}} + (x_v,y_v))$ are {\CM} modules of dimension $n+t-1$ and $n$, respectively, we have that $H^i(S/(J_{G\setminus\{v\}}+(x_v,y_v)))=0$ for all $i\neq {n+t-1}$ and $H^i(S/(J_{K_{n-1}} + (x_v,y_v)))=0$ for all $i\neq n$.
	Moreover, since $\depth S/J_G = n+1$ (see \cite{EHH}) and $\dim S/J_G = n+t-1$, we have that $H^i(S/J)=0$ for all $i<n+1$ or $i>n+t-1$. Merging all this informations gives us the short exact sequence
	\begin{equation} \label{1:BEI-BlockStar-Seq-Hn}
		0 \to H^{n}(S/(J_{K_{n-1}}+(x_v,y_v))) \to H^{n+1}(S/J_{G}) \to H^{n+1}(S/J_{K_n}) \to 0
	\end{equation}
	and the isomorphism
	\begin{equation}\label{1:BEI-BlockStar-Seq-Canonical}
		H^{n+t-1}(S/J_G) \isom H^{n+t-1}(S/J_{G\setminus\{v\}}+(x_v,y_v)).
	\end{equation}
	Moreover, $H^i (S/J_G)=0$, for all $i\notin \{n+1,n+t-1\}$. This implies condition (3) of the Theorem; in fact, by Local Duality, we have
	\[
	\omega^i (S/J_G)=0,\qquad \text{for all $i\notin \{n+1,n+t-1\}$},
	\]
	as stated. By applying Local Duality to sequence \eqref{1:BEI-BlockStar-Seq-Hn}, we obtain the short exact sequence on the modules of deficiency,
	\begin{equation*}
		0 \to \omega^{n+1}(S/J_{K_n}) \to \omega^{n+1}(S/J_{G}) \to \omega^{n}(S/(J_{K_{n-1}}+(x_v,y_v)) \to 0,
	\end{equation*}
	which, by applying again the local cohomology functor and dualizing, leads to the short exact sequence
	\begin{align*}
		0 \to \omega^{n+1}(\omega^{n+1}(S/J_G)) &\to \omega^{n+1}(\omega^{n+1}(S/J_{K_n})) \to \\
		&\to \omega^{n}(\omega^{n}(J_{K_{n-1}}+(x_v,y_v))) \to \omega^{n}(\omega^{n+1}(S/J_G)) \to 0.
	\end{align*}
	Note that we took into account that $\omega^{n+1}(S/J_{K_n})$ is a {\CM} module of dimension $n+1$ and $\omega^n(S/(J_{K_{n-1}}+(x_v,y_v))$ is a {\CM} module of dimension $n$. Moreover, by the Depth Lemma, we obtain $\depth \omega^{n+1}(S/J_G)\geq n$. Consider now the map $\omega^{n+1}(\omega^{n+1}(S/J_{K_n})) \to \omega^{n}(\omega^{n}(S/(J_{K_{n-1}} +(x_v,y_v) )))$, then there exist an epimorphism $f$ which lifts the map and makes the following diagram
	\begin{equation*}
	\begin{gathered}
		\xymatrixcolsep{1.25cm}
		\xymatrix{
		  S/J_{K_n} \ar[r]^-{f} \ar[d] & S/(J_{K_{n-1}}+(x_v,y_v)) \ar[d]\\
		  \omega(\omega(S/J_{K_n})) \ar[r] & \omega(\omega(S/(J_{K_{n-1}}+(x_v,y_v)))}
	\end{gathered}
	\end{equation*}
	commutative. Observe that we wrote $\omega(\omega(S/J_{K_n}))$ instead of $\omega^{n+1}(\omega^{n+1}(S/J_{K_n}))$ and $\omega(\omega(S/(J_{K_{n-1}}+(x_v,y_v))))$ instead of $\omega^{n}(\omega^{n}(S/(J_{K_{n-1}}+(x_v,y_v))))$ since we are dealing with canonical modules. Moreover, by Proposition \ref{1:DeficiencyModules-Prop}, the vertical maps are isomorphisms. Then we obtain the short exact sequence,
	\begin{equation}\label{Eq:BlockStar_FinalSeq}
		\begin{aligned}
		0 \to \omega^{n+1}(\omega^{n+1}(S/J_G)) \to S/J_{K_n} & \xrightarrow{f} S/(J_{K_{n-1}}+(x_v,y_v)) \\
			&\to \omega^{n}(\omega^{n+1}(S/J_G)) \to 0.
		\end{aligned}
	\end{equation}
	Since $f$ is a surjective map, then $\omega^{n}(\omega^{n+1}(S/J_G))=0$. Hence, $\depth \omega^{n+1}(S/J_G) = \dim \omega^{n+1}(S/J_G) = n+1$. It follows that $\omega^{n+1}(S/J_G)$ is a {\CM} module of dimension $n+1$. Moreover, the sequence \eqref{Eq:BlockStar_FinalSeq} induces the isomorphism\linebreak
	$\omega^{n+1}(\omega^{n+1}(S/J_G))\isom(J_{K_{n-1}}+(x_v,y_v))/J_{K_n}$ and this completes the proof of point 1. of the statement.
	
	Finally, the isomorphism \eqref{1:BEI-BlockStar-Seq-Canonical} leads to $\omega(S/J_G) \isom \omega(S/J_{G\setminus\{v\}}+(x_v,y_v))$, hence $\omega(S/J_G)$ is a {\CM} module of dimension $n+t-1$ and point 2. of the statement is also proved.
	
	Lastly, by conditions 1. and 2. and Theorem \ref{1:SCM_DefModChar}, it easily follows that $S/J_G$ is {\SCM}.
\end{proof}

In the previous result, we have a complete description of the modules of deficiency of $S/J_G$, when $G$ is a block star graph. Nevertheless, it can be difficult to apply the same techniques of Theorem \ref{1:BlockStar_DefModDescr+SCM} in general cases. In fact, by sequence \eqref{eq:GeneralMayerVietoris}, it is not always possible to get informations on the non-vanishing local cohomology modules, hence for the modules of deficiency (see for instance \cite{W}).

\begin{Example}
	Let $G$ be the tree on 6 vertices as shown in the picture below.
	
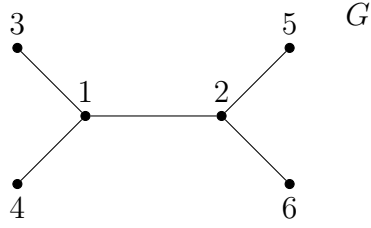
\begin{figure}[H]
		\centering
		\begin{tikzpicture}[scale=.9,vertices/.style={draw,black,fill=black,circle,inner sep=1.2pt}]
			\node[vertices] (1) at (-1,0) {};
			\node at (-1,10pt) {$1$};
			\node[vertices] (2) at (1,0) {};
			\node at (1,10pt) {$2$};
			
			\node[vertices] (3) at (-2,1) {};
			\node at (-2,1cm+10pt) {$3$};
			\node[vertices] (4) at (-2,-1) {};
			\node at (-2,-1cm-10pt) {$4$};
			\node[vertices] (5) at (2,1) {};
			\node at (2,1cm+10pt) {$5$};
			\node[vertices] (6) at (2,-1) {};
			\node at (2,-1cm-10pt) {$6$};
			
			\node at (3,1.5) {$G$};

			\draw (1)--(2) (1)--(3) (1)--(4) (2)--(5) (2)--(6);
		\end{tikzpicture}
		\label{fig:not_vanish}
	\end{figure}

Such a graph has the following cutsets
\[
\emptyset, \{1\}, \{2\}, \{1,2\},
\]
hence sequence \eqref{eq:GeneralMayerVietoris}, by choosing $v=1$, becomes
\[
0\to S/J_G \to S/J_{G_1} \dirsum S/(J_{G \setminus \{1\}}+(x_1,y_1)) \to S/(J_{G_1 \setminus \{1\}}+(x_1,y_1)) \to 0.
\]
One has $\depth S/J_G =7$ and $\dim S/J_G=8$. Moreover, $\dim S/J_{G_1} =\linebreak \dim S/(J_{G\setminus\{1\}}+(x_1,y_1))=8$ and $\dim S/(J_{G_1\setminus\{1\}}+(x_1,y_1))=7$
By applying local cohomology functors, we obtain the sequence
\begin{equation}\label{eq:Example_NotVanishing}
\begin{split}
	0\to H^{6}(S/(J_{G_1 \setminus \{1\}}+(x_1,y_1))) \to H^7(S/J_G) \to H^7(S/J_{G_1}) \to \\
	\to H^7(S/(J_{G_1 \setminus \{1\}}+(x_1,y_1))) \to H^8(S/J_G) \to \\
	\to H^8(S/J_{G_1}) \dirsum H^8(S/(J_{G \setminus \{1\}}+(x_1,y_1)))\to 0,
\end{split}
\end{equation}
where we also kept in mind that $\depth S/J_{G_1}=7$, $\depth S/(J_{G\setminus\{1\}}+(x_1,y_1))=8$ and $\depth S/(J_{G_1\setminus\{1\}}+(x_1,y_1))=6$.

The local cohomology modules appearing in sequence \eqref{eq:Example_NotVanishing} are all non--zero, hence it is difficult to retrieve informations about $H^7(S/J_G)$ and $H^8(S/J_G)$ as for the proof of Theorem \ref{1:BlockStar_DefModDescr+SCM}.
\end{Example}

For this reason, we will use different strategies in Section \ref{sec:3}, in order to prove that all block graphs have {\SCM} binomial edge ideal.

More recently, Goodarzi \cite{G} provided an interesting method to construct the dimension filtration of a homogeneous ideal.

\begin{Definition}
	Let $I\subset S$ be a homogeneous ideal, with $d=\dim S/I$, and let $I=\Sect_{j=1}^r Q_j$ be the minimal primary decomposition of $I$. For all $1\leq j\leq r$, let $P_j = \sqrt{Q_j}$ be the radical of $Q_j$. For all $-1\leq i\leq d$, denote by $I^{<i>}$ the ideal
	\[
	I^{<i>} = \Sect_{\dim S/{P_j}>i} Q_{j},
	\]
	where $I^{<-1>}=I$ and $I^{<d>}=S$. Thus, we have a filtration
	\[
	I = I^{<-1>} \subseteq I^{<0>} \subseteq I^{<1>} \subseteq \ldots \subseteq I^{<d-1>} \subseteq I^{<d>} = S.
	\]
	Moreover, if we rewrite the above sequence of inclusions modulo $I$, we have
	\begin{equation}
		0 \subseteq I^{<0>}/I \subseteq I^{<1>}/I \subseteq \ldots \subseteq I^{<d-1>}/I \subseteq S/I,
	\end{equation}
	which is the dimension filtration of the $S$-module $S/I$.
\end{Definition}

By the previous definition it is possible to characterize the {\SCM} property (see \cite[Proposition 16]{G}).

\begin{Proposition}\label{prop:SCM-char-goodarzi}
	Let $I\subset S$ be a homogeneous ideal and suppose $d=\dim S/I$. Then $S/I$ is {\SCM} if and only if
	\[
	\depth S/{I^{<i>}} \geq i+1,\qquad\text{for all $0\leq i < d$}.
	\]
\end{Proposition}

The above characterization can be reformulated, using the above notation, in term of
\[
\dD(I) = \left\{\dim S/P_j \ :\ P_j=\sqrt{Q}_j\right\} = \{d_1,\ldots,d_{\ell}\}.
\]

\begin{Lemma}\label{lem:DG}
	Let $I\subset S$ be a homogeneous ideal and $\dD(I)=\{d_1,\ldots,d_{\ell}\}$, with $d_1 < \ldots < d_{\ell}=d$. Then $S/I$ is {\SCM} if and only if
	\[
	\depth S/{I^{<d_j-1>}} = d_j,\qquad\text{for all $j=1,\ldots,\ell$}.
	\]
\end{Lemma}
\begin{proof}
	Suppose $S/I$ is {\SCM}, then by Proposition \ref{prop:SCM-char-goodarzi}, we have that
	\[
	\depth S/{I^{<i>}} \geq i+1,\qquad\text{for all $0\leq i < d$},
	\]
	in particular
	\[
	\depth S/{I^{<d_j-1>}} \geq d_j,\qquad\text{for all $j=1,\ldots,\ell$}.
	\]
	Moreover, for all $j=1,\ldots,\ell$, we have $\depth S/{I^{<d_j-1>}}\leq \dim S/{I^{<d_j-1>}}$ and, by the definition of $I^{<d_j-1>}$, by \cite[Proposition 1.2.13]{BH_book} that $\depth S/{I^{<d_j-1>}}\leq d_j$. Therefore, for all $j=1,\ldots,\ell$, we have
	\[
	\depth S/{I^{<d_j-1>}} = d_j.
	\]
	
	Conversely, suppose
	\[
	\depth S/{I^{<d_j-1>}} = d_j,\qquad\text{for all $j=1,\ldots,\ell$}
	\]
	and show that $S/I$ is {\SCM}. Let $0\leq i < d$. We must distinguish two cases:
	
	{\sc I Case:} suppose $i=d_{j}-1$ for some $j$, then $\depth S/I^{<i>}=d_j=i+1$;
	
	{\sc II Case:} suppose $d_j \leq i < d_{j+1}$ for some $j$, then $\depth S/I^{<i>}=d_{j+1}\geq i+1$.\\
	Therefore, for all $0\leq i < d$,
	\[
	\depth S/I^{<i>} \geq i+1,
	\]
	and $S/I$ is {\SCM}.
\end{proof}


We end this section by providing some sufficient conditions and some necessary conditions for a graph $S/J_G$ to be {\SCM}. For this aim the next notion will be crucial.

A graph $G$ on $n$ vertices is {\em $\ell$-vertex-connected} (or simply $\ell$-connected) if $\ell < n$ and for every subset $T\subset V(G)$ of vertices such that $|T| < \ell$, the graph $G\setminus T$ is connected. The {\em vertex-connectivity} (or simply {\em connectivity}) of $G$, denoted by $\kappa(G)$, is the maximum integer $\ell$ such that $G$ is $\ell$-vertex-connected.

In \cite{BNB}, the authors proved the following result.

\begin{Theorem}{\em (\cite[Theorem 3.19]{BNB})}\label{thm:Banerjee-Depth-Connectivity}
	Let $G$ be a non--complete, connected graph on $n$ vertices and $J_G$ the corresponding binomial edge ideal. Then
	\[
	\depth S/{J_G} \leq n - \kappa(G) + 2,
	\]
	where $\kappa(G)$ is the connectivity of the graph $G$.
\end{Theorem}

The next Proposition characterizes all {\SCM} graphs with only one non--empty cutset.

\begin{Proposition}\label{prop:onecut}
	Let $G$ be a graph such that $\cC(G)=\{\emptyset, T\}$, let $t=|T|$ and let $c=c(T)$ denote the number of connected components of $G\setminus T$.	Then $\depth S/J_G=n-t+2$, and $J_G$ is {\SCM} in the following cases:
	\begin{enumerate}[label={\normalfont(\roman*)}]
		\item $t=1$;
		\item $t\geq 2$ and $c=2$;
	\end{enumerate}
\end{Proposition}
\begin{proof}
	From $\cC(G)=\{\emptyset, T\}$ we have
	\[
	J_G=P_{\emptyset} \cap P_T.
	\]
	We recall that $P_\emptyset=J_{K_n}$ and we observe that since $(\{x_v,y_v\}_{v \in T}) \subset P_T$, then $P_\emptyset+P_T=J_{K_{n-t}}+(\{x_v,y_v\}_{v \in T})$.
	Let $X_T=\{x_1,\ldots,x_n,y_1 \ldots y_n\} \setminus \{\{x_v,y_v\} : v \in T\}$.
	Hence, we have the Mayer-Vietoris sequence 
	\[
	0 \to S/J_G \to S/{J_{K_n}} \dirsum S/{P_T} \to {K}[X_T]/J_{K_{n-t}} \to 0.
	\]
	In particular, from Depth Lemma
	\begin{align*}
	\depth S/J_G & \geq \min\{n+1,n-t+c,(n-t+1)+1\}\\
				& = n-t+2, 
	\end{align*}
	because $t\geq 1$ and $c\geq 2$. Since $G$ has connectivity $t$, from Theorem \ref{thm:Banerjee-Depth-Connectivity} we have
	\[
	\depth S/J_G \leq n-t+2, 
	\] 
	thus $\depth S/J_G = n-t+2$. Therefore, $J_G$ is {\SCM} if for $i\geq n-t+2$ we have $J_G^{<i>} \neq J_G$, namely
	\[
	J_G^{<i>} \in \{J_{P_T},J_{K_n}\}.
	\]
	To achieve the above condition, it is sufficient to require that $J_G^{<n-t+2>} \in \{J_{P_T},J_{K_n}\}$. To obtain 
	\[
	J_G^{<n-t+2>}=J_{P_T},
	\]
	we should have $\dim S/J_{K_n}=n+1\leq n-t+2$. 
	Similarly, to obtain 
	\[
	J_G^{<n-t+2>}=J_{K_n},	
	\] 
	we should have $n-t+c\leq n-t+2$.
	To sum up, we have the conditions
	\begin{enumerate}[label={\normalfont(\roman*)}]
		\item $n-t+2 \geq  n+1 \implies t\leq 1$,
		\item $n-t+2 \geq n-t+c \implies c\leq 2$,
	\end{enumerate}
	and the assertion follows.
\end{proof}

The next Lemma illustrates a necessary condition for the binomial edge ideal $J_G$ to be {\SCM}.

\begin{Lemma}\label{lemma:necessary-SCM}
	Let $G$ be a non--complete, connected graph on $n$ vertices. Let $\overline{T} \in \cC(G)$ such that $\dim S/P_{\overline{T}}=\m(G)$. If $J_G$ is {\SCM}, then 
	\[
	\kappa(G)-|\overline{T}| + c(\overline{T}) \leq 2,
	\]
	where $\kappa(G)$ is the connectivity of the graph $G$.
\end{Lemma}
\begin{proof}
	We verify the statement by contradiction. Suppose
	\[
	\kappa(G)-|\overline{T}| + c(\overline{T}) \geq 3,
	\]
	we must prove that $J_G$ is not {\SCM}.	We proceed in the following way:
	\begin{gather*}
		\kappa(G)-|\overline{T}| + c(\overline{T}) \geq 3\ \  \Implies\ \  n-|\overline{T}| + c(\overline{T}) \geq n - \kappa(G) + 3 \\ 
		{}\Implies n - |\overline{T}| + c(\overline{T}) \geq (n - \kappa(G) + 2) + 1
	\end{gather*}
	Since $\dim S/P_{\overline{T}} = n - |\overline{T}| + c(\overline{T})$, it is clear that
	\[
	J_G^{<n - \kappa(G) + 2>} = J_G.
	\]
	Hence, we obtain
	\begin{align*}
		\depth S/J_G^{<n - \kappa(G) + 2>} 	& = \depth J_G \\
											& \leq n - \kappa(G) + 2,
	\end{align*}
	which, by Proposition \ref{prop:SCM-char-goodarzi}, contradicts the fact that $J_G$ is {\SCM}.
\end{proof}

As a consequence of Lemma \ref{lemma:necessary-SCM}, we are able to prove the following:

\begin{Proposition}\label{prop:SCM-mindim_onecutvertex}
	Let $G$ be a non--complete, connected graph on $n$ vertices. Suppose $\m(G)=\dim S/P_{\emptyset}=n+1$. 
	If $J_G$ is {\SCM}, then $G$ has at least one cutpoint.
\end{Proposition}
\begin{proof}
	By Lemma \ref{lemma:necessary-SCM}, we know that
	\[
	\kappa(G)-|\overline{T}| + c(\overline{T}) \leq 2,
	\]
	where $\overline{T}=\emptyset$ and $c(\overline{T})=1$. Hence,
	\[
	\kappa(G) + 1 \leq 2\ \Implies\ \kappa(G) \leq 1.
	\]
	Since $G$ is a connected graph, necessarily $\kappa(G) = 1$ and the assertion follows.
\end{proof}

Inspired by Proposition \ref{prop:SCM-mindim_onecutvertex}, in Section \ref{sec:3}, we will study block graphs. In fact, they satisfy the hypothesis of the Proposition \ref{prop:SCM-mindim_onecutvertex} recursively, by removing cutpoints.
	
Observe that being {\SCM} and with at least a cutpoint does not imply that $m(G)=n+1$ (e.g. Example \ref{ex:SCM_onecutpoint_m(G)<n+1}).

\section{Cycles and wheels}\label{sec:2}
In this section we focus on some families of graphs with no cutpoints: we prove that cycles and wheels have {\SCM} binomial edge ideal. We recall that a module $M$ is called \emph{almost Cohen-Macaulay} if $\depth M  \geq \dim M - 1 $, and $M$ is said \emph{approximately Cohen-Macaulay} (introduced in \cite{Goto}) if and only if $M$ is almost Cohen-Macaulay and {\SCM} (see \cite[Proposition 4.5]{S}). Let $C_n$ be a cycle with $n$ vertices. We first remark that cycle graphs are not {\CM} in general if $n\geq 4$. It has been proved by Zafar in \cite{Z} that $\depth S/J_{C_n} = n$ (see \cite[Theorem 4.5]{Z}). Since, $\dim S/J_{C_n}=n+1$,  then $S/J_{C_n}$ is almost Cohen-Macaulay. In the paper \cite{Z}, Zafar also proves that cycles are approximately Cohen-Macaulay, and hence {\SCM}. Nevertheless, we provide a shorter proof for the sequential Cohen-Macaulayness, by using an elementary property of cycles. Without loss of generality, from now on, we will assume $n\geq 4$.

\begin{Lemma}\label{lemma:cycle-cutsets}
	Let $C_n$ be a cycle with $n$ vertices. Then the nonempty cutsets of $C_n$ have cardinalities at least $2$. Moreover, if $T\neq \emptyset$ is a cutset then $|T|$ is equal to the number of connected components of $G\setminus T$. 
\end{Lemma}
\begin{proof}
	Let $V(C_n)=\{1,\ldots,n\}$, and $E(C_n)=\{\{1,2\},\{2,3\},\ldots,\{n-1,n\}\}\cup\{\{1,n\}\}$. Since $C_n$ is biconnected we need at least $2$ vertices to break the cycle. Moreover, without loss of generality we may assume that a nonempty cutset contains the vertex 1. We claim that $T=\{i_1=1,\ldots,i_r\}$ with $i_k>i_{k-1}+1$ for all $2\leq k\leq r$, and $i_r<n$. That is $T$ is a cutset if and only if it is a subset of $\{1,\ldots, n\}$ of cardinality at least $2$ and such that there are no adjacent vertices in $T$. Hence, $G\setminus T$ is a collection of $|T|$ paths and isolated vertices as stated.
\end{proof}

\begin{Proposition}\label{prop:SCM_cycle}
	Let $C_n$ be a cycle with $n$ vertices. Then $J_{C_n}$ is {\SCM}.
\end{Proposition}
\begin{proof}
	We first observe that, by Lemma \ref{lemma:cycle-cutsets}, it easily follows that
	\[
	\dim S/J_{C_n} = n+1.
	\]
	In fact, given a cutset $T$ of $C_n$, if $T=\emptyset$, then $\dim S/P_{\emptyset}=n+1$. Otherwise, if $T\neq\emptyset$, by Lemma \ref{lemma:cycle-cutsets} we have that $c(T)=\lvert T \rvert$, thus $\dim S/P_T=n-\lvert T \rvert + c(T) = n$. Therefore,
	\[
	\dim S/J_{C_n} = \max\{n,n+1\} = n+1,
	\]
	as stated.
	
	Our aim is to apply Lemma \ref{lem:DG} to prove that $S/J_{C_n}$ is {\SCM}. We observe that $\dD(J_G)=\{n,n+1\}$, and $S/J_{C_n}$ is {\SCM} if 
	\[
	\depth	S/{J_{C_n}}^{<n-1>}=n \mbox{ and } \depth	S/{J_{C_n}}^{<n>}=n+1.
	\]
	From ${J_{C_n}}^{<n-1>}=J_{C_n}$ and $S/{J_{C_n}}^{<n>}=P_{\emptyset}$, the assertion follows.
\end{proof}

In order to prove a consequence of the previous result, we recall that if $G$ is a graph on the vertex set $V(G)=[n]$ and $v\notin V(G)$, the {\em cone} of $v$ on $G$, denoted by $\Cone(v,G)$, is the graph on the vertex set $V(G)\union\{v\}$ and edge set $E(G)\union\{\{v,w\}\ :\ w\in V(G)\}$. The {\em wheel graph} is $W_n=\Cone(n+1,C_n)$. Moreover, any nonempty cutset $T$ of $W_n$ is $T' \cup \{n+1\}$ where $T'$ is a nonempty cutset of $C_n$. Hence, $c(T)=|T|-1$.
It is known that (see \cite[Corollary 3.11]{KS})
\[
\depth S/J_{W_n} = n.
\]

\begin{Corollary}\label{cor:SCM_wheel}
	Let $W_n$ be a wheel graph. Then $J_{W_n}$ is {\SCM}.
\end{Corollary}
\begin{proof}
	With a similar arguing to that of Proposition \ref{prop:SCM_cycle}, we obtain
	\[
	\dim S/J_{W_n} = \max\{n,n+2\} = n+2.
	\]
	As done in Proposition \ref{prop:SCM_cycle}, by applying Lemma \ref{lem:DG}, the assertion follows.
\end{proof}

From the proofs of Proposition \ref{prop:SCM_cycle} and Corollary \ref{cor:SCM_wheel}, it arises that $\m (C_n)=n$ and $\m(W_n)=(n+1)-1$.

The class of blocks that are (non-){\SCM} is a big one, as the following and easy example shows.
 
\begin{Example} Let $G$ be a graph with a unique non-empty cutset of cardinality greater than $1$, then $G$ is a block. By Proposition \ref{prop:onecut} and its notation, if $c=2$ we construct an infinite family of graphs that are {\SCM}. If $c>2$ we have an infinite family of graphs that are non-{\SCM}.
\end{Example}

\section{Block graphs}\label{sec:3}
In this Section we prove that all block graphs have {\SCM} binomial edge ideals. As observed in \cite[Remark 1.7]{ERT2022b}, we reduce to the study of indecomposable connected block graphs. Moreover, for the sake of simplicity we set $J_v:=J_{G_v}$ and $J_{\overline{v}}:=J_{G\setminus \{v\}}$ in the Mayer--Vietoris sequence \eqref{eq:GeneralMayerVietoris}.

\begin{Lemma}\label{NiceLemma}
	Let $v$ a cutpoint of an indecomposable connected block graph $G$ such that 
	\[
	G\setminus \{v\}= G_1\sqcup B_1\sqcup \ldots \sqcup B_r
	\]
	with $r \geq 2$, where $B_i$ are complete graphs. Then
	\[
	J_v^{<i>}+(J_{\overline{v}}+(x_v,y_v))^{<i>}=(J_{v}+(x_v,y_v))^{<i-1>}.
	\]
\end{Lemma}
\begin{proof}
	We claim that if $G$ is an indecomposable block graph, then\linebreak $\cC(G_v)=\cC(G\setminus \{v\})$. We first observe that the inclusion $\cC(G\setminus \{v\})\subseteq \cC(G_v)$ easily holds. Thus, it suffices to show that $\cC(G_v)\subseteq\cC(G\setminus \{v\})$. Moreover, if $T$ is a cutset of $G$ that does not intersect $N_{G_1}(v)$, then $T$ is obviously a cutset of both graphs. Let $T$ be a cutset of $G_v$ and assume $T$ is not cutset of $G\setminus\{v\}$. From the above observation, $T \cap N_{G_1}(v)\neq\emptyset$ and hence there exists $u\in T \cap N_{G_1}(v)$ such that $c_{G_1}(T\setminus\{u\})=c_{G_1}(T)$. That is, $u$ is not a cutpoint, namely $u$ is free in the graph $G_1$. This implies  $N_{G}(u)=N_{G_1}(u) \cup \{v\}$ and $G$ is decomposable in $u$, a contradiction.

	Observe now that the block containing $v$ in $G_v$, namely $B_v$, is a block of $G_v$ itself, such that $|V(B_v)|\geq 4$ and having at least $3$ free vertices. Thanks to this we have that
	\[
	J_v^{<i>}+(x_v,y_v)=(J_{v}+(x_v,y_v))^{<i-1>}.
	\]
	In fact, 
	\[
	J_v^{<i>}=\bigcap_{\dim S/P_T>i} P_T=(\bigcap P'_T)+J_{F_v},
	\]
	where $F_v$ is the complete graph containing the free vertices of $B_v$, and
	\[
	J_v^{<i>}+(x_v,y_v)=(\bigcap P'_T)+J_{F_v\setminus \{v\}}=J_v^{<i-1>}.
	\]
	Hence, it is sufficient to prove that for all $i$
	\[
	J_v^{<i>}+(x_v,y_v)\supseteq (J_{\overline{v}}+(x_v,y_v))^{<i>}.
	\]
	For any $T \in \cC(G_v)=\cC(G\setminus \{v\})$ let $P_T^v$ and $P_T^{\overline{v}}$ be primes of $J_v$ and $J_{\overline{v}}$, respectively. Then
	\[
	\dim P_{T}^v=n-|T|+c_{G_v}(T) \mbox{  and   } 	\dim P_{T}^{\overline{v}}=(n-1)-|T|+c_{G\setminus \{v\}}(T).
	\]
	Observe that $c_{G\setminus \{v\}}(T)=c_{G_v}(T) + r$. Since $r\geq 2$, then 
	\[
	\dim P_{T}^{\overline{v}}=(n-1)-|T|+c_{G_v}(T) +r =\dim P_{T}^v +r-1.
	\]
	Let $\dD(J_{G_v})=\{d_0,\ldots,d_l\}$ and $\dD(J_{G\setminus \{v\}})=\{\overline{d}_0,\ldots,\overline{d}_l\}$ where $d_0=\dim P_{\emptyset}=n+1$. We have proved that $d_k=\overline{d}_k-r+1$.
	we show that for any $k$ 
	\[
	J_v^{<d_k-1>}+(x_v,y_v)\supseteq (J_{\overline{v}}+(x_v,y_v))^{<d_k-1>}.
	\]
	
	If $k=0$, and $d_{0}-1 < \overline{d}_{0}-1 $, then
	\[
	J_v^{<d_0-1>}+(x_v,y_v)=J_v+(x_v,y_v)\supseteq (J_{\overline{v}}+(x_v,y_v))^{<d_0-1>}=J_{\overline{v}}+(x_v,y_v)
	\]
	because there is no dimension shift, {\em i.e.} the filtered ideals $(J_{\overline{v}}+(x_v,y_v))^{<i>}$, for $i=0,1,\ldots,d_0-1$, coincide with $J_{\overline{v}}+(x_v,y_v)$ itself. 

	Let $0<k \leq l$, we observe that since $G\setminus v$ is not connected then
	\[
	(J_{\overline{v}}+(x_v,y_v))^{<\overline{d}_k-1>}=\left(\Sect_{\dim S/P'_T\geq \overline{d}_k-r} P'_T\right)+\sum_{i=1}^r J_{B_i}+(x_v,y_v).
	\]
	Since $J_{F_v}$ is the binomial edge ideal induced by the free vertices of the block containing $v$ in $G_v$, then
	\[
	J_{v}^{<d_k-1>}=\Sect_{\dim S/P_T>d_{k}} P_T=\left(\Sect_{\dim S/P'_T\geq d_k-1} P'_T\right)+ J_{F_v}.
	\]
	Since $\sum_{i=1}^r J_{B_i}\subset J_{F_v}$ and $\overline{d}_{k}-r=d_k-1$, we have 
	\[
	J_v^{<d_k-1>}+(x_v,y_v)\supseteq (J_{\overline{v}}+(x_v,y_v))^{<\overline{d}_k-1>} \supseteq (J_{\overline{v}}+(x_v,y_v))^{<d_k-1>}
	\]	
	and the assertion follows.
\end{proof}

Whilst Lemma \ref{NiceLemma} holds for block graphs, it is not true in general, as showed in the next example.

\begin{Example}
	Let $G$ be the graph on $7$ vertices obtained by attaching a vertex to the complete bipartite graph $K_{2,4}$ as shown in the figure below.
	\begin{figure}[H]
		\centering
		\begin{tikzpicture}[scale=.9,vertices/.style={draw,black,fill=black,circle,inner sep=1pt}]
			\foreach \i in {1,2}{%
				\node[vertices] (\i) at (\i+3,0) {};
				\node at (\i+3,10pt) {$\i$}; 
			}
			\foreach \j in {3,4,5,6}{%
				\node[vertices] (\j) at (\j,-1.5cm) {};
				\node at (\j,-1.5cm-10pt) {$\j$}; 
			}
			\node[vertices] (7) at (7,-0.5cm) {};
			\node[above right] at (7) {$7$};
			
			\foreach \i in {1,2}{%
				\foreach \j in {3,4,5,6}{%
					\draw (\i)--(\j);
				}
			}
			\draw (6)--(7);
		\end{tikzpicture}
		\label{fig:countex}\caption{The graph $K_{2,4}$ with a whisker}
	\end{figure}
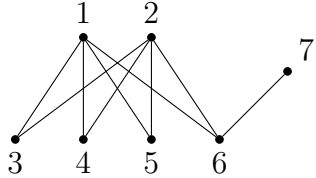
	
	Such a graph has the cutpoint $6$. By a direct computation, using {\tt Macaulay2} \cite{GDS}, one can show that
	\[
	J_{G_{6}}^{<7>} + (J_{G\setminus\{6\}}+(x_6,y_6))^{<7>} \neq (J_{G_{6}\setminus\{6\}}+(x_6,y_6))^{<6>}.
	\]
\end{Example}

We are now in position to prove the main result.

\begin{Theorem}
	Let $G$ be an indecomposable, connected block graph. Then $J_G$ is {\SCM}.
\end{Theorem}
\begin{proof}
	Our aim is to apply Goodarzi's characterization (see \cite[Proposition 2.6]{G}) to $S/J_G$, proving that for all $0 \leq i < d=\dim S/J_G$,
	\[
	\depth (S/J_{G}^{<i>}) \geq i+1.
	\]
	
	We proceed by induction on the number $r$ of the cutpoints of $G$. If $r=1$, then
	$G$ is a block star graph. Let $v$ be the unique cutpoint of $G$ and let $c_v=c\left(G\setminus \{v\}\right)$. Hence,
	\[
	J_G = J_{v} \sect (J_{\overline{v}} + (x_v,y_v)),
	\]
	and $d=\dim S/J_G = \dim S/J_{\overline{v}} = n + c_v - 1$. Moreover, it is clear that
	\[
	J_G^{<i>} = \begin{cases}
				J_G &\quad \text{if $0\leq i\leq n$},\\
				(J_{\overline{v}} + (x_v,y_v)) &\quad \text{if $n < i < d$},\\
				S &\quad \text{if $i=d$}.
				\end{cases}
	\]
	Therefore, if $0\leq i\leq n$, we have
	\[
	\depth(S/J_G^{<i>}) = n + 1 \geq i+1.
	\]
	Otherwise, if $n < i < d$
	\begin{align*}
	\depth(S/J_G^{<i>}) & = \depth(S/(J_{\overline{v}} + (x_v,y_v))) \\
						& = \dim(S/(J_{\overline{v}} + (x_v,y_v))) \\
						& = d \\
						& \geq i+1,
	\end{align*}
	and this proves the base case.
	
	Suppose now $r>1$ and assume the statement true for $r-1$. Consider the Mayer–Vietoris sequence
	\begin{equation*}
		0 \to S/J_G \to S/{J_v} \dirsum S/{J_{\overline{v}}} \to S/(J_v + (J_{\overline{v}}+(x_v,y_v))) \to 0.
	\end{equation*}
	Fix now $i$ such that $0 \leq i < d$, the above sequence induces the following sequence on the filtered ideals:
	\begin{equation}\label{eq:MayerVietorisFiltered}
		0 \to S/{J_G^{<i>}} \to S/{J_v^{<i>}} \dirsum S/({J_{\overline{v}}}+(x_v,y_v))^{<i>} \to S/(J_v^{<i>} + (J_{\overline{v}}+(x_v,y_v))^{<i>}).
	\end{equation}
	In order to apply the Depth Lemma to sequence \eqref{eq:MayerVietorisFiltered}, we express the ideal\linebreak $J_v^{<i>} + (J_{\overline{v}}+(x_v,y_v))^{<i>}$ in a more suitable form. By Lemma \ref{NiceLemma}, we know that
	\begin{equation*}
		J_v^{<i>} + (J_{\overline{v}}+(x_v,y_v))^{<i>} = (J_v + (x_v,y_v))^{<i-1>},
	\end{equation*}
	hence sequence \eqref{eq:MayerVietorisFiltered} becomes
	\begin{equation}\label{eq:MayerVietorisFilteredAfterLemma}
		0 \to S/{J_G^{<i>}} \to S/{J_v^{<i>}} \dirsum S/({J_{\overline{v}}}+(x_v,y_v))^{<i>} \to (J_v + (x_v,y_v))^{<i-1>},
	\end{equation}	
	and, by applying the Depth Lemma to the above sequence, we obtain
	\begin{align*}
		\depth(S/J_G^{<i>}) & \geq \min\big\{\depth(S/{J_{v}^{<i>}} \dirsum S/({J_{\overline{v}}}+(x_v,y_v))^{<i>}),\\
		&\phantom{\geq\min\big\{.}\depth(S/(J_v + (x_v,y_v))^{<i-1>}) + 1 \big\}.
	\end{align*}
	Observe that $\depth S/{J_v^{<i>}} \leq \depth S/({J_{\overline{v}}}+(x_v,y_v))^{<i>}$, since $G_{v}$ is a block graph and $G\setminus \{v\}$ is tensor product of block graphs. Therefore,
	\[
	\depth(S/{J_v^{<i>}} \dirsum S/({J_{\overline{v}}}+(x_v,y_v))^{<i>}) = \depth S/{J_v^{<i>}}.
	\]
	Then
	\[
	\depth(S/J_G^{<i>}) \geq \min\big\{ \depth S/{J_v^{<i>}}, \depth(S/(J_v + (x_v,y_v))^{<i-1>}) + 1 \big\}.
	\]
	By inductive hypothesis, we have that
	\[
	\depth S/{J_v^{<i>}} \geq i+1,\qquad \depth(S/(J_v + (x_v,y_v))^{<i-1>}) \geq i,
	\]
	and the assertion follows.
\end{proof}

\section{Cone graphs}\label{sec:4}
Cone graphs have been studied by many authors (e.g. \cite{BMS2018},\cite{KS},\cite{RR}), to control the behavior of the depth and the {\CM} property.

In this section we study the binomial edge ideals of cone graphs and we specify under which conditions they are {\SCM}.

\begin{Theorem}\label{thm:cones}
	Let $G = G_1 \sqcup G_2 \sqcup \ldots \sqcup G_r$ be a graph, where $G_i$ is a connected graph for each $i=1,\ldots,r$, with $r\geq 2$. Let $v$ be a vertex such that $v\notin V(G)$ and $H=\Cone(v,G)$. Then $J_{H}$ is {\SCM} if and only if $J_{G_1},\ldots,J_{G_r}$ are {\SCM}.
\end{Theorem}
\begin{proof}
	We first observe that, since $H$ is a cone graph, then
	\[
	H_{v}=K_{n+1},\qquad H \setminus\{v\}=G.
	\]
	Hence, we write
	\[
	J_H = J_{K_{n+1}} \sect (J_G + (x_v,y_v)).
	\]
	Moreover,
	\[
	S/(J_G+(x_v,y_v)) \isom S_1/J_{G_1} \otimes \cdots\otimes S_r/J_{G_r},
	\]
	where $S_i=K[x_u,y_u\ :\ u\in V(G_i)]$. Thus, $S_i/J_{G_i}$ are {\SCM} if and only if
	$S/(J_G+(x_v,y_v))$ is {\SCM}.

	Since $J_H^{<i>} = J_{K_{n+1}}^{<i>} \sect (J_G + (x_v,y_v))^{<i>}$ and $\dim S/J_{K_{n+1}} = n+2$, for $i\geq n+2$
	\[
	J_H^{<i>} = (J_G + (x_v,y_v))^{<i>}.
	\]
  	Therefore, if $i\geq n+2$, we have $\depth(S/J_{H}^{<i>})\geq i+1$ if and only if\linebreak $\depth(S/(J_G+(x_v,y_v))^{<i>})\geq i+1$.
	
	Thus, we now deal with the case $i\leq n+1$. We observe that 
	\[
	\dim S/P_{\{v\}}=n+r \geq n+2,
	\]
	thus $P_{\{v\}}$ appears in the primary decomposition of any $(J_G + (x_v,y_v))^{<i>}$ for any $i\leq n+1$. Then we write
	\[
	\begin{aligned}
	J_{H}^{<i>} & = J_{H}^{<i>} \cap P_{\{v\}} \\
				  &	= P_{\emptyset} \cap (J_{G}+(x_v,y_v))^{<i>} \cap P_{\{v\}}\\
				  & = (P_{\emptyset}\cap P_{\{v\}}) \cap (J_{G}+(x_v,y_v))^{<i>}.
	\end{aligned}
	\]
	Note that $P_\emptyset \cap P_{\{v\}}= J_{B_v}$, where $B_v$ is a block star graph with cutpoint $v$. Now, we consider the Mayer-Vietoris sequence
    \begin{equation}\label{eq:Mayer-Vietoris-Cone-i}
	0 \to S/J_H^{<i>} \to S/J_{B_v}  \dirsum S/(J_{G}+(x_v,y_v))^{<i>} \to S/J_{B_v}+(J_G+(x_v,y_v))^{<i>} \to 0.
	\end{equation}
	Since
	\[
	(J_G + (x_v,y_v))^{<i>} \subset P_{\{v\}},
	\]
    and any cutset of $H$ contains $v$, that is
    \[
 	(x_v,y_v) \subset (J_G + (x_v,y_v))^{<i>}.
    \]
    Therefore,
    \[
   	P_{\{v\}}=J_{B_v}+(x_v,y_v)\subseteq  J_{B_v}+(J_G + (x_v,y_v))^{<i>}\subseteq    J_{B_v}+P_{\{v\}} = P_{\{v\}},
    \]
    yielding $J_{B_v}+(J_G + (x_v,y_v))^{<i>}= P_{\{v\}}$.
	Hence, by the Depth Lemma applied to \eqref{eq:Mayer-Vietoris-Cone-i}, we have
	{\small \begin{align*}
	\depth(S/J_{H}^{<i>}) & \geq \min\big\{\depth(S/J_{B_v}),\depth(S/(J_G+(x_v,y_v))^{<i>}),\depth(S/P_{\{v\}})+1\big\}\\
		& = \min\big\{n+2,\depth(S/(J_G+(x_v,y_v))^{<i>}),n+r+1\big\}
	\end{align*}}
	and 
	{\small \begin{align*}
	\min\big\{\depth(S/(J_G+(x_v,y_v))^{<i>}),\depth(S/J_{B_v})\big\} & \geq \min\big\{\depth(S/J_H^{<i>}),\depth(S/P_{\{v\}})\big\}\\
		& = \min\big\{\depth(S/J_H^{<i>}),n+r\big\}
	\end{align*}}
	Since $i+1\leq n+2$, we have that 
	\[
	\depth S/{J_H^{<i>}} \geq i+1  \iff \depth S/{(J_G+(x_v,y_v))^{<i>}} \geq i+1,
	\]
	and the assertion follows.
\end{proof}

\begin{Remark}
	The hypothesis $r\geq 2$ of Theorem \ref{thm:cones} is fundamental. Indeed, even if the wheel $W_n$ is a cone on a cycle $C_n$, and they are both {\SCM}, there are examples of {\SCM} graphs $G$ such that the cone from a vertex $v$ to $G$ is not {\SCM}.
\end{Remark}

\begin{Example}
	Let $G=K_{1,3}$ be the claw graph on vertices $\{1,2,3,4\}$, then $H=\Cone(5,G)$ is such that $\cC(G)=\{ \emptyset, T=\{1,5\} \}$ and $c(T)=3$ (see Figure \ref{fig:badcone}(A),\ref{fig:badcone}(B)). From Theorem \ref{1:BlockStar_DefModDescr+SCM} $G$ is {\SCM}, while from Proposition \ref{prop:onecut}, $H$ is not {\SCM}.
	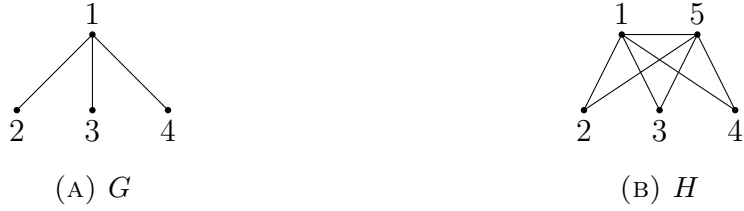
\begin{figure}[H]
		\centering
		\begin{subfigure}{0.5 \textwidth}
			\centering
			\begin{tikzpicture}
				\filldraw (0,0) circle (1pt) node [anchor=north]{2};
				\filldraw (1,0) circle (1pt) node [anchor=north]{3};
				\filldraw (2,0) circle (1pt) node [anchor=north]{4};
				\filldraw (1,1) circle (1pt) node [anchor=south]{1};
				
				\draw (1,1)--(0,0);
				\draw (1,1)--(1,0);
				\draw (1,1)--(2,0);
			\end{tikzpicture}
			\caption{$G$}
		\end{subfigure}%
		\begin{subfigure}{0.5 \textwidth}
			\centering
			\begin{tikzpicture}
				\filldraw (0,0) circle (1pt) node [anchor=north]{2};
				\filldraw (1,0) circle (1pt) node [anchor=north]{3};
				\filldraw (2,0) circle (1pt) node [anchor=north]{4};
				\filldraw (0.5,1) circle (1pt) node [anchor=south]{1};
				\filldraw (1.5,1) circle (1pt) node [anchor=south]{5};
				
				\draw (0.5,1)--(0,0);
				\draw (0.5,1)--(1,0);
				\draw (0.5,1)--(2,0);
				\draw (1.5,1)--(0,0);
				\draw (1.5,1)--(1,0);
				\draw (1.5,1)--(2,0);
				\draw (1.5,1)--(0.5,1);
			\end{tikzpicture}
			\caption{$H$}
		\end{subfigure}%
		\caption{A cone $H$ on a connected graph $G$}\label{fig:badcone}
	\end{figure}
\end{Example}

Theorem \ref{thm:cones} is also useful to construct a {\SCM} graph $G$ with a cutpoint such that $\m(G)\neq |V(G)|+1$.

\begin{Example}\label{ex:SCM_onecutpoint_m(G)<n+1}
	Let $G=\Cone(9,C_4\sqcup C_4)$ be the cone on the graph obtained as the disjoint union of two cycle $C_4$ (see Figure \ref{fig:ConeC4}). Such a graph is {\SCM} by Theorem \ref{thm:cones} and has cutpoint $9$. Nevertheless, $\m(G)=8< 10 = |V(G)|+1$.
	\begin{figure}[ht]
		\centering
		\begin{tikzpicture}[vertices/.style={draw,black,fill=black,circle,inner sep=1pt}]
			\node[vertices] (1) at (-2,1) {};
			\node[above] at (1) {$1$};
			
			\node[vertices] (2) at (-1,0) {};
			\node[above right] at (2) {$2$};
			
			\node[vertices] (3) at (-2,-1) {};
			\node[below] at (3) {$3$};
			
			\node[vertices] (4) at (-3,0) {};
			\node[above left] at (4) {$4$};
			
			\node[vertices] (5) at (2,1) {};
			\node[above] at (5) {$5$};
			
			\node[vertices] (6) at (3,0) {};
			\node[above right] at (6) {$6$};
			
			\node[vertices] (7) at (2,-1) {};
			\node[below] at (7) {$7$};
			
			\node[vertices] (8) at (1,0) {};
			\node[above left] at (8) {$8$};

			\node[vertices] (9) at (0,-.55) {};
			\node[below] at (9) {$9$};

			\draw (1)--(2)--(3)--(4)--(1)--cycle;
			\draw (5)--(6)--(7)--(8)--(5)--cycle;
			\foreach \k in {1,...,8}{%
				\draw (\k)--(9);
			}
		\end{tikzpicture}
		\caption{A cone graph on two graphs $C_4$}\label{fig:ConeC4}
	\end{figure}
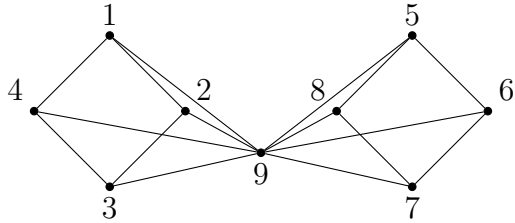
\end{Example}

All trees, and all (even) cycles have {\SCM} binomial edge ideals. Moreover, the only complete bipartite graphs are $C_4$ \cite{SZ} and trees (Section \ref{sec:3}). A complete characterization of bipartite {\CM} graphs has been done in \cite{BMS2018}. Hence, the following natural question arises.

\begin{Question}
	What are bipartite graphs with {\SCM} binomial edge ideals?
\end{Question}


\end{document}